\documentclass[12pt,letterpaper]{article}
\usepackage[a4paper, total={7in, 9in}]{geometry}
\usepackage{amsmath,amsthm,amssymb}
\usepackage{enumerate}
\usepackage[dvipsnames]{xcolor}
\usepackage{tikz}

\def\F{\mathcal{F}}
\def\G{\mathcal{G}}
\def\HH{\mathcal{H}}

\def\LL{\mathcal{L}}

\def\S{\mathcal{S}}
\def\P{\mathbb{P}}

\newtheorem{thm}{Theorem}[section]
\newtheorem{conj}[thm]{Conjecture}
\newtheorem{prop}[thm]{Proposition}
\newtheorem{claim}[thm]{Claim}
\newtheorem{lemma}[thm]{Lemma}
\newtheorem{cor}[thm]{Corollary}

\newtheorem{defi}[thm]{Definition}
\newtheorem{rem}[thm]{Remark}

\newcommand{\floor}[1]{\left\lfloor#1\right\rfloor}

\newcommand \track[1] {{\color{black} #1}}
\newcommand{\st}[1]{{\iffalse #1 \fi}}

\begin{document}

\title{On the number of $P$-free set systems for tree posets $P$}
\author{J{\'o}zsef Balogh,  Ramon I.  Garcia, and Michael C. Wigal\thanks{corresponding author}}
\date{Department of Mathematics\\
University of Illinois at Urbana-Champaign\\
Urbana, IL, USA\\
\small{ \texttt{\{jobal, rig2, wigal\}@illinois.edu}}\\
\bigskip
\today}

\maketitle

\begin{abstract}
    We say a finite poset $P$ is a {\it tree poset} if its Hasse diagram is a tree. Let $k$ be the length of the largest chain contained in $P$. We show that when $P$ is a fixed tree poset, the number of $P$-free set systems in $2^{[n]}$ is $2^{(1+o(1))(k-1){n \choose \lfloor n/2\rfloor}}$. The proof uses a generalization of a theorem by Boris Bukh 
     together with a variation of the multiphase graph container algorithm.\\

     \noindent \textbf{Keywords} Set System, Tree Poset, Forbidden Subposet, Supersaturation, Containers
\end{abstract}
\section{Introduction}

Given two posets $P$ and $Q$, a \textit{poset homomorphism} is a function $f:P\rightarrow Q$ such that $f(A)\leq f(B)$ whenever $A\leq B$. A poset $Q$ \textit{contains} a poset $P$ if there is an injective poset homomorphism $\pi:P\rightarrow Q$. If $Q$ does not contain $P$ we say that $Q$ is $P$-\textit{free}. All posets considered in this paper are finite.

Given a poset $P$ and two elements $x,z\in P$, we say that $x$ \textit{covers} $z$ if $x>z $ and there is no $y\in P$ with $x> y> z$.  We define the \textit{Hasse diagram} of $P$, denoted $\st{H}\track{\text{Ha}}(P)$, as a graph with vertex set $P$ drawn in the plane such that we draw an edge from $x$ to $y$ upwards only if $y$ covers $x$ (we allow edges to cross in the drawing). A poset $P$ is a \textit{tree poset} if $\st{H}\track{\text{Ha}}(P)$ is a tree. The \textit{height} of a poset is the length of the longest maximal chain of $P$.

For a fixed poset $P$ and positive integer $n$, define the size of a largest $P$-free set system in $2^{[n]}$ as $\text{La}(n,P)$. The systematic study of $\text{La}(n,P)$, began with the work of Katona and Tarj\'{a}n \cite{KatonaTarjan}. The following conjecture has been central to the study of $\text{La}(n,P)$ which has been formulated in many places with slight variations, see \cite{Bukh,GerbnerNagyPatkosVizer,GriggsLiLu}. For positive integers $n$ and $x < n$, we let $\binom{[n]}{x}$ denote the family of all subsets of $2^{[n]}$ of size $x$. 

\begin{conj}\label{conj:LA}
    Let $P$ be a poset, then 
    $$\text{La}(n,P) = e(P)\binom{n}{\lfloor n/2 \rfloor}(1 + o(1)),$$
    where $e(P)$ denotes the largest integer $\ell$ such that for all $j$ and $n$ the family $\bigcup_{i = 1}^{\ell} \binom{[n]}{i+j}$ is $P$-free. 
\end{conj}
Conjecture~\ref{conj:LA} can be viewed as a broad generalization of the classical Sperner's Theorem \cite{Sperner}. The smallest unresolved case for Conjecture~\ref{conj:LA} is the 2-dimensional Boolean lattice (i.e., the diamond poset), which has received significant attention in the literature, see \cite{AxenovichManskeMartin,GriggsLiLu,GroszMethukuTompkins,KramerMartinYoung}. 
Only very recently, Conjecture~\ref{conj:LA} has been disproven in the case when $P$ is the $d$-dimensional Boolean lattice for $d \ge 4$, see \cite{EllisIvanLeader}. For a broader overview of the history of Conjecture~\ref{conj:LA}, see the survey of Griggs and Li \cite{GriggsLi} or the textbook of Gerb\track{n}er and Patk{\'o}s \cite[Chapter 7]{GerberPatkos}. As $e(P) = k - 1$ for all tree posets $P$ of height $k$, the following theorem of Bukh proves Conjecture~\ref{conj:LA} for all tree posets.

\begin{thm}\label{thm:bukh}\cite[Theorem 1]{Bukh}
    Let $P$ be a tree poset of height $k$, then
    $$\text{La}(n,P)=(k-1)\binom{n}{\lfloor n/2\rfloor }\left(1+O(1/n)\right).$$
\end{thm}

Motivated by Conjecture~\ref{conj:LA}, Gerbner, Nagy, Patk{\'o}s, and Vizer \cite{GerbnerNagyPatkosVizer} conjectured the following. 

\begin{conj}\label{conj:supersat}
    For every   poset $P$ the number of $P$-free set systems in $2^{[n]}$ is
    $$2^{(1+o(1))\st{\text{La}}\track{e}(n,P)\binom{n}{\lfloor n/2\rfloor }}.$$
\end{conj}

Patk{\'o}s and Treglown \cite{PatkosTreglown} proved Conjecture~\ref{conj:supersat} in the special case when $P$ is a tree poset of height at most five and radius at most 2. Our main result is proving Conjecture~\ref{conj:supersat} for all tree posets. 

\begin{thm}\label{thm:main}
Let $P$ be a tree poset of height $k$, then the number of $P$-free set systems in $2^{[n]}$ is 
$$2^{(1+o(1))(k-1)\binom{n}{\lfloor n/2\rfloor}}.$$
\end{thm}

Boris Bukh \cite{Bukh} observed that Theorem \ref{thm:bukh} also extends to a larger class of posets that embed into trees in an ideal way. This same observation extends to our Theorem \ref{thm:main} as well. Recall that $e(P)$ denotes the largest integer $\ell$ such that for all $j$ and $n$ the family $\bigcup_{i = 1}^{\ell} \binom{[n]}{i+j}$ is $P$-free.  

\begin{cor}\label{cor:main}
    Let $P$ be a poset and $Q$ be a tree poset of height $k$ such that $e(P) = k - 1$ and $Q$ contains $P$. Then the number of $P$-free set systems in $2^{[n]}$ is
    $$ 2^{(1 + o(1))(k-1)\binom{n}{\lfloor n/2 \rfloor}}.$$
\end{cor}

\begin{proof}
    As $e(P) = k - 1$ we have that $\bigcup_{i = 0}^{k-2} \binom{[n]}{\lfloor  n/2 \rfloor +i}$ and each of its subfamilies is $P$-free. This implies that the number of $P$-free set systems in $2^{[n]}$ is at least $2^{(1+o(1))(k-1)\binom{n}{\lfloor n/2\rfloor}}$. By Theorem \ref{thm:main}, the number of $Q$-free is at most $2^{(1+o(1))(k-1)\binom{n}{\lfloor n/2\rfloor}}$. As $Q$ contains $P$, the corollary follows.
\end{proof}

One of the smallest nontrivial posets Corollary~\ref{cor:main} applies to is the butterfly poset $B$, see Figure~\ref{fig:butterfly}. As $e(B) = 2$ and $B$ is contained in the $X$ poset, see Figure~\ref{fig:butterfly}, by Corollary \ref{cor:main}, we have that Conjecture \ref{conj:supersat} is true for the butterfly poset as well.

\begin{figure}[h]
    \centering
    \begin{tikzpicture}
        \node[shape=circle,draw=black,fill=black] (A) at (0,2) {};
        \node[shape=circle,draw=black, ,fill=black] (B) at (0,0) {};
        \node[shape=circle,draw=black, ,fill=black] (C) at (2,2) {};
        \node[shape=circle,draw=black, ,fill=black] (D) at (2,0) {};

        \foreach \from/\to in {A/B,A/D, D/C, C/B}
    \draw (\from) -- (\to);
    \end{tikzpicture}\quad\quad\quad\quad\quad
       \begin{tikzpicture}
        \node[shape=circle,draw=black,fill=black] (A) at (0,2){};
        \node[shape=circle,draw=black, ,fill=black] (B) at (0,0){};
        \node[shape=circle,draw=black, ,fill=black] (C) at (2,2){};
        \node[shape=circle,draw=black, ,fill=black] (D) at (2,0){};
        \node[shape=circle,draw=black, ,fill=black] (E) at (1,1){};

        \foreach \from/\to in {A/D, C/B}
    \draw (\from) -- (\to);
    \end{tikzpicture}
    \caption{Hasse diagrams for the butterfly poset and the $X$ poset, respectively.}\label{fig:butterfly}
\end{figure}
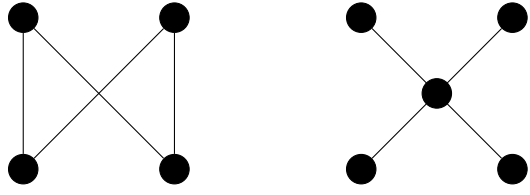

There are two natural potential extensions of our results. One is to extend to $P$-free posets, where we forbid $P$ as an induced subposet, see the relevant conjecture of Gerbner, Nagy,   Patk\'{o}s,
               and Vizer~\cite[Conjecture~9]{GerbnerNagyPatkosVizer} or the extensions of~\cite{Bukh} by Boehnlein and Jiang~\cite{TaoJiangInduced}.
The second is to investigate random variants of our results, see for example \cite{PatkosTreglown}, which would require proving stronger supersaturation conditions.

The rest of the paper is organized as follows. In Section \ref{sec:Prel}, we present some preliminary results. Then, we proceed to prove our supersaturation results Corollaries \ref{cor:supersat1} and \ref{cor:supersat2}. We then prove our main container result, Lemma \ref{lem:container}. Finally, we prove Theorem \ref{thm:main} in Section \ref{sec:Proof}.

\section{Preliminaries}\label{sec:Prel}

We use the following standard upper bound on the sum of binomial coefficients, \track{see e.g., \cite{galvin2014three}.}

\begin{prop}\label{prop:Entropy}
    For every $\alpha\in [0,1/2]$ and $n\in \mathbb{Z}^{+}$,
    \[
    \sum_{i\leq \alpha n}\binom{n}{i}\leq 2^{H(\alpha) n},
    \]
    where $H(p)=-p\log_2 p-(1-p)\log_2 (1-p)$, is the binary  entropy function. 
\end{prop}

The easy fact that a removal of a leaf vertex of a tree results in a tree, implies the following claim.

\begin{claim}\label{claim:orderP}
For every tree poset $P$ of size $m$ and an element $x\in P$, there is a total ordering $\prec$ of the elements of $P$, which we write as $x=x_1\prec \ldots\prec x_{m}$ and such that for each $i \in [m]$, the induced subgraph $\st{H}\track{\text{Ha}}(P)[x_1,\ldots, x_{i}]$ is a tree such that $x_i$ has degree 1.
\end{claim}

\begin{defi}\label{def:Blowup}
For a tree poset $P$, an element $x\in P$, a total ordering of $P$ beginning at $x$ as described in Claim \ref{claim:orderP}, and a positive integer $t$, define the $t$-\textit{blowup of } $P$ \textit{centered at }$x$, denoted $P(x,t)$, to be the poset constructed as follows. Replace each $x_i\in P$ with $t^{d(x_i)}$ copies of $x_i$, where $d(x_i):=d(x,x_i)$ is the distance from $x$ to $x_i$ in the Hasse diagram $\st{H}\track{\text{Ha}}(P)$. Label the copies of $x_i$ as $x_{i,1},\ldots, x_{i,t^{d(x_i)}}$.

For each $i>1$, there is exactly one edge $x_ix_j$ in $\st{H}\track{\text{Ha}}(\st{p}\track{P})$ such that $d(x_i)=d(x_j)+1$ and $i>j$. We split the copies of $x_{i}$ into $t^{d(x_i)-1}$ sets $V_{i,k}=\{x_{i,(k-1)t+1},\ldots, x_{i, kt}\}$ for $k\in [t^{d(x_i)-1}]$. When $x_i$ covers $x_j$ in $P$, we set $x_{j,k}\track{<_{P(x,t)}}\st{>_{P(x,t)}}v$ for all $v\in V_{i,k}$. Otherwise, $x_j$ covers $x_i$ in $P$, and we set $x_{j,k}\track{>_{P(x,t)}}\st{<_{P(x,t)}} v$ for all $v\in V_{i,k}$.
\end{defi}
Notice that although there are possibly multiple choices for the order as described in Claim \ref{claim:orderP}, the poset $P(x,t)$ is unique up to isomorphism. Furthermore, the indexing of the elements of $P(x,t)$ induces a total ordering given by the lexicographic ordering.   
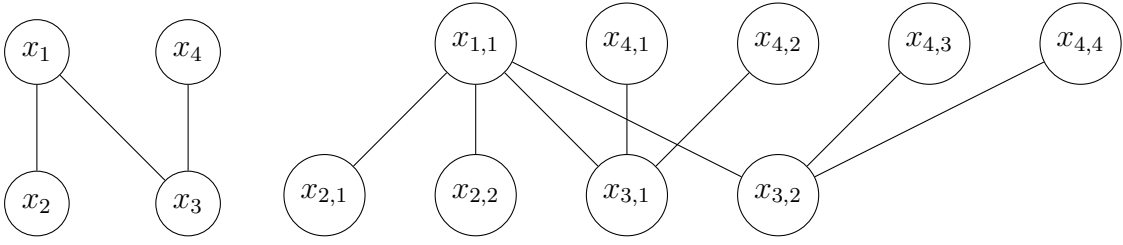
\begin{figure}[h]
    \centering
    \begin{tikzpicture}
        \node[shape=circle,draw=black] (A) at (0,2) {$x_{1}$};
        \node[shape=circle,draw=black] (B) at (0,0) {$x_{2}$};
        \node[shape=circle,draw=black] (C) at (2,2) {$x_{4}$};
        \node[shape=circle,draw=black] (D) at (2,0) {$x_{3}$};

        \foreach \from/\to in {A/B,A/D, D/C}
    \draw (\from) -- (\to);
    \end{tikzpicture}\quad \quad
       \begin{tikzpicture}
        \node[shape=circle,draw=black] (A) at (0,2) {$x_{1,1}$};
        \node[shape=circle,draw=black] (B) at (-2,0) {$x_{2,1}$};
        \node[shape=circle,draw=black] (C) at (0,0) {$x_{2,2}$};
        \node[shape=circle,draw=black] (D) at (2,2) {$x_{4,1}$};
        \node[shape=circle,draw=black] (E) at (4,2) {$x_{4,2}$};
        \node[shape=circle,draw=black] (F) at (2,0) {$x_{3,1}$};
        \node[shape=circle,draw=black] (G) at (4,0) {$x_{3,2}$};
         \node[shape=circle,draw=black] (H) at (6,2) {$x_{4,3}$};
        \node[shape=circle,draw=black] (I) at (8,2) {$x_{4,4}$};

        \foreach \from/\to in {A/B,A/C,A/F,A/G, F/D,F/E,G/H,G/I}
    \draw (\from) -- (\to);
    \end{tikzpicture}
    \caption{$2$-blowup of a path poset with four vertices.}
\end{figure}

\subsection{Supersaturation tools}

Unless otherwise stated, we follow the notation and terminology of West \cite{West} with regard to partial orders. We say a finite poset is \emph{graded} if every maximal chain has the same size. Thus our notion of graded poset of height $k$ corresponds to the notion of $k$-saturated posets of Bukh \cite{Bukh}. We have changed this terminology to avoid potential conflict with the notion of a $k$-saturated chain decomposition with respect to the classical Greene-Kleitman Theorem \cite{GreeneKleitman}. Thus in a graded tree poset, every leaf in the Hasse diagram must be a minimal or maximal element. The following lemma shows that every tree poset is contained in a graded tree poset. As observed in \cite{Bukh}, working with graded posets is simpler.

\begin{lemma}\cite[Lemma 5]{Bukh}\label{lem:kRankedContention}
   Let $P$ be a tree poset of height $k$. Then there exists a graded tree poset $\hat{P}$ of height $k$ such that $P$ is an induced subposet of $\hat{P}$ and $|\hat{P}| \le s k$, where $s$ denotes the number of maximal chains of $P$. 
\end{lemma}

\begin{rem}
    The inequality $|\hat{P}| \le s k$ is not in the original statement of \cite[Lemma 5]{Bukh} but can be read out of its proof. 
\end{rem} 

\begin{lemma}\label{lem:blowup_graded}
    Let $P$ be a graded tree poset of height $k$, $x \in P$ and $a$ is some positive integer. Then $P(x,a)$ is a graded tree poset of height $k$.
\end{lemma}

\begin{proof}
    Clearly, $P$ has height $k$ if and only if $P(x,a)$ has height $k$. Further, every maximal chain of $P(x,a)$ corresponds to a maximal chain of $P$ in the natural way, hence $P(x,a)$ is also graded. Furthermore, by its definition, the Hasse diagram of $P(x,a)$ is also a tree.
\end{proof}

An \emph{interval} of a poset $P$ is a set of the form $\{z \in P : x \le z \le y\}$ where $x,y \in P$. It is easy to check that for tree posets, an interval is either the empty set or a chain.

\begin{lemma}\label{lem:boris_decomposition}\cite[Lemma 6]{Bukh}
    Let $P$ be a graded tree poset of height $k$ such that $P$ is not a chain. Then there is an element $v\in P$, which is a leaf in $\st{H}\track{\text{Ha}}(P)$, and an interval $I$ of length $|I|\leq k-1$ containing $v$ such that $\st{H}\track{\text{Ha}}(P\setminus I)$ is a tree, and the poset $P\setminus I$ is a graded poset of height $k$.  
\end{lemma}

The following lemma is introduced to correct a minor mistake in the induction statement of Bukh's original proof \cite[Lemma 7]{Bukh}. 

\begin{lemma}\label{lem:decomposition}
    Suppose $P$ is a graded tree poset of height $k$. Then there exists an integer $\ell \ge 1$ and maximal chains $C_1,\ldots, C_{\ell}$ of $P$ such that
    \begin{enumerate}[(i)]
        \item For all $1 \le j \le \ell$, $\bigcup_{i = 1}^j C_i$ is a graded poset of height $k$ and $\bigcup_{i = 1}^{\ell} C_i = P$.
        \item For all $1 < j \le \ell$, $I_{j} := C_{j}\setminus \left(\bigcup_{i = 1}^{j-1} C_i \right)$ is a nonempty interval containing a minimal or maximal element of $P$.
        \item For all $1 < j \le \ell$, $C_{j} \setminus I_{j} \subseteq C_i$ for some $1 \le i < j$.
    \end{enumerate}
\end{lemma}

\begin{proof}
    Let $m = |P|$. If $P$ is a chain, then the lemma holds trivially. Thus we may suppose $P$ is not a chain and $m > 2$. We proceed with induction on $m$. 

    As $P$ is not a chain and $\st{H}\track{\text{Ha}}(P)$ is a tree, by Lemma~\ref{lem:boris_decomposition}, there exists some $v \in P$, such that $v$ is a leaf in $\st{H}\track{\text{Ha}}(P)$ and there exists an interval $I$ of length $|I| \le k-1$ containing $v$ such that $\st{H}\track{\text{Ha}}(P \setminus I)$ is a tree and the poset $P \setminus I$ is a graded poset of height $k$. As we may consider the dual of $P$, without loss of generality, we may suppose $v$ is a minimal element of $P$. By induction, there exists $s \ge 1$ and chains $C_1,\ldots, C_{s}$ of $P \setminus I$ such that (i), (ii) and (iii) hold.
    
   As both $\st{H}\track{\text{Ha}}(P)$ and $\st{H}\track{\text{Ha}}(P\setminus I)$ are trees, $P$ is graded and of height $k$, and $|I| \le k - 1$, there is exactly one element of $u \in P \setminus I$,   which covers the maximum element of $I$. We let $I_{s + 1} := I$. To complete the proof, we must find a maximal chain $C_{s + 1}$ of $P$ such that $I_{s + 1} \subseteq C_{s  + 1}$ and (i), (ii), and (iii) holds for $C_1, \ldots, C_{s},C_{s+1}$.

    Letting $I_1 = C_1$, by (i) and (ii), we have
    $$ \bigcup_{i = 1}^{s} I_i = \bigcup_{i = 1}^{s} C_i = P \setminus I_{s + 1}.$$
    By (ii),  there exists a unique $j$ such that $u \in I_j$. Define
    \[ C_{s + 1} := I_{s+1} \cup \{x \ge u : x \in C_j\}.\]
    As $u \in I_j \subseteq C_j$, by induction, $C_j$ contains a maximal element of $P$. Then the set $\{x \ge u: x \in C_j\}$ is an interval between $u$ and a maximal element of $P$. Furthermore, as $I_{s + 1}$ is an interval containing a minimal element of $P$, and $u$ covers the maximum element of $I_{s + 1}$, we have that $C_{s + 1}$ is a maximal chain of $P$. As $P$ is a graded poset of height $k$ and 
    $$ \bigcup_{i = 1}^{s + 1} C_i = P,$$
    we have that (i) holds. We have that (ii) holds by construction of $I_{s + 1}$, and as $C_{s + 1} \setminus I_{s + 1} \subseteq C_j$, we have that (iii) holds as well. Thus the lemma holds by induction.
\end{proof}

Let $P$ be a graded tree poset, and $\{C_1,\ldots,C_{\ell}\}$ be a set of maximal chains satisfying conditions (i), (ii), and (iii) of Lemma~\ref{lem:decomposition}. We say $\{C_1,\ldots,C_{\ell}\}$ is a \emph{graded 
chain cover} of $P$. We note that as $P$ is graded, $|C_i| = k$ for all $i \in [\ell]$.

We define a $(k,a)$-\textit{marked chain} as an ordered pair $(M,\{F_1,\ldots,F_k\})$ such that $M$ is a maximal chain in $2^{[n]}$ and $F_i \in M$ for all $i \in [k]$ such that $F_1\supsetneq F_2\supsetneq F\st{_{2}}\track{_{3}}\supsetneq \ldots \supsetneq F_k$  and $|F_i\setminus F_{i+1}|\geq a$ for all $i\in [k-1]$. We call the $F_i$'s the \textit{markers}. The following lemma was proved in the special case of $(k,1)$-marked chains by Boris Bukh \cite[Lemma 4]{Bukh}. 

\begin{lemma}\label{lem:k_marked_chains_lowerbound}
Let $k$ and $a$ be positive integers and let $\varepsilon > 0$. If $\F\subseteq 2^{[n]}$ is of size $$|\mathcal{F}|>((k-1)a+\varepsilon)\binom{n}{\lfloor n/2\rfloor},$$ 
then there are at least $\frac{\varepsilon}{k}n!$ $(k,a)$-marked chains, whose markers are in $\F$.
\end{lemma}
\begin{proof}
    Let $D_i$ denote the number of maximal chains that contain exactly $i$ elements from $\F$. By double counting pairs $(M,F)$ such that $M$ is a maximal chain in $2^{[n]}$ and $F\in M\cap \F$, we obtain
    \begin{equation}\label{eqn:k_marked_chains_1}
           \sum_{i}iD_i=\sum_{F\in \F}\frac{n!}{\binom{n}{|F|}}\geq |\F|\frac{n!}{\binom{n}{\lfloor n/2 \rfloor}}\geq ((k-1)a+\varepsilon)n!.
    \end{equation}
    Trivially we also have, 
    \begin{equation}\label{eqn:k_marked_chains_2}
        \sum_{i}D_i=n!.
    \end{equation}
    Using \eqref{eqn:k_marked_chains_1} and \eqref{eqn:k_marked_chains_2}, we have the following inequality.
    \begin{equation}\label{eqn:k_marked_chains_3}
         \sum_i iD_i - (k-1)a\sum_{i}D_i \ge \varepsilon n!. 
    \end{equation} 
    The number of binary strings of length $i$ with $k$ ones such that each of the first $k-1$ ones is followed by at least $a-1$ zeros is $\binom{i-(k-1)(a-1)}{k}$.  It follows that the number of $(k,a)$-marked chains is at least,
    \begin{align*}
    &\sum_{i\geq (k-1)a+1}\binom{i-(k-1)(a-1)}{k}D_i = \sum_{i\geq (k-1)a+1}\binom{i-(k-1)(a-1)-1}{k-1}\frac{i-(k-1)(a-1)}{k}D_i\\
 &\geq \sum_{i\geq (k-1)a+1}\frac{i-(k-1)(a-1)}{k}D_i \geq \sum_{i} \frac{i}{k}D_i-\sum_{i}\frac{(k-1)a}{k}D_i \geq \frac{\varepsilon}{k}n!,
    \end{align*}
where the last relation follows from \eqref{eqn:k_marked_chains_3}.
    
\end{proof}

Let $P$ be a graded tree poset, let $\{C_1,\ldots, C_{\ell}\}$ be a graded chain cover of $P$, let $\F \subseteq 2^{[n]}$ be a set system, and let $a > 0$ be a positive integer. We are interested in finding an embedding $\pi$ of $P$ into $\F$. Furthermore, by Lemma~\ref{lem:k_marked_chains_lowerbound}, there exists a large family $\LL$ of $(k,a)$-marked chains whose markers belong to $\F$. For induction purposes, we also require that our embedding maps each $C_i$ to the set of markers of some $(k,a)$-marked chain belonging to $\LL$. 

The proof of Lemma~\ref{lem:AuxSupersat1} is relatively straightforward, and follows closely the proof of \cite[Lemma~7]{Bukh}. We iteratively embed subsequentially larger subposets of $P$ with respect to a graded chain cover $\{C_1,\ldots,C_{\ell}\}$ of $P$. At each stage, we clean the respective family of $(k,a)$-marked chains $\LL$ of the ``bad" chains. We have to ensure that we do not delete too many chains, as otherwise our current choice of embedding may fail to extend to all of $P$.

\begin{lemma}\label{lem:AuxSupersat1}
    Let $\alpha \in (0,1/2)$, let $a$ and $n$ be positive integers, let $c = a! \alpha^{-a}$. Let $P$ be a graded tree poset of height $k$ and let  $\{C_1,\ldots, C_{\ell}\}$ be a graded chain cover of $P$. Let $\F \subseteq 2^{[n]}$ be a set system such that all sets of $\F$ are of size between $\alpha n$ and $(1-\alpha)n$ and let $K \in [n]$ such that no maximal chain of $2^{[n]}$ contains more than $K \ge 2k$ sets from $\F$. Let $\LL$ be a family of $(k,a)$-marked chains  with markers in $\F$ such that
    \[
    |\LL| \ge \frac{c }{n^a}\binom{K}{k}^2\binom{|P|+1}{2}n!.
    \]
    Finally, suppose $|P| < K^{-k} n^{a}/c$. Then there is an embedding $\pi$ of $P$ into $\F$ such that for every $i \in [\ell]$, $\pi(C_i)$ is the set of markers of some $(k,a)$-marked chain in $\LL$.
\end{lemma}

\begin{proof}

The proof is by induction on $\ell$. If $\ell = 1$, then $P$ is a chain of height $k$ with $C_1 = P$. In this case, as $\LL \neq \emptyset$, we may choose an arbitrary $(k,a)$-marked chain $M \in \LL$. As the markers of $M$ belong to $\F$, there is a natural embedding of $P$ into the markers of $M$.

Suppose $\ell \ge 2$. Let $I = C_{\ell} \setminus \left( \bigcup_{i = 1}^{\ell - 1}C_i \right)$. By Lemma~\ref{lem:decomposition}, $I$ is a nonempty interval containing a minimal or maximal element of $P$ such that $C_\ell \setminus I \subseteq C_j$ for some $1 \le j < \ell$. Without loss of generality, we may assume $I$ contains a minimal element of $P$. Let $s = |C_{\ell} \setminus I\st{_{\ell}}| > 0$. 

 The chain $F_1 \supsetneq \cdots \supsetneq F_s$ is a \emph{bottleneck} (with respect to $\F$, $P$, and $\LL$) if there exists 
\[ \S \subseteq \{X: X \subseteq F_s, \text{ } X\in\F \text{ and } |X|\leq |F_s|-a\},  \]
 such that $|\S| \le |P|$, and for every $(k,a)$-marked chain of $\LL$ of the form $(M,\{F_1,\ldots,F_k\})$, such that $s < k$ and $F_i \st{\subsetneq}\track{\supsetneq} F_{i+1}$ for all $i \in [k-1]$,
 we have $\S \cap \{F_{s+1},\ldots, F_k\} \neq \emptyset$. If such an $\S$ exists, we say $\S$ is a \emph{witness} of that the chain $\{F_1,\ldots,F_s\}$ is a bottleneck. If $\{F_1,\ldots,F_s\}$ is a bottleneck, then we let $\S(\{F_1,\ldots,F_s\})$ be an arbitrary but fixed witness of $\{F_1,\ldots,F_s\}$. We say a $(k,a)$-marked chain $(M,\{F_1, \ldots, F_k\})$ is \emph{bad} if $\{F_1,\ldots,F_s\}$ is a bottleneck, otherwise it is \emph{good}.

Our goal is to find an $\LL' \subseteq \LL$ such that if we embed $P \setminus I$ with respect to $\{C_1,\ldots,C_{\ell - 1}\}$ with $(k,a)$-chains belonging to $\LL'$, then the set of markers corresponding to $C_\ell \setminus I$ is not a bottleneck. Thus it suffices to construct an $\LL' \subseteq \LL$ which contains no bad $(k,a)$-marked chains of $\LL$. 

Let $R \subseteq [K]$ such that $|R| = s = |C_{\ell} \setminus I|$. We sample a maximal chain $M$ of $2^{[n]}$ uniformly at random.  Suppose $|M \cap \F| = t > 0$, and let $F_1 \supsetneq \cdots \supsetneq F_t$ such that $F_i \in M \cap \F$ for every $i \in [t]$. Then we let $C_R(M)$ denote the function
\begin{align*}
C_R(M) = \begin{cases}
    \{ F_i : i \in R\} &\text{ if } R \subseteq [t];\\
    \emptyset &\text{ otherwise.}
    \end{cases}
\end{align*}
 That is, for $R \subseteq [t]$, the function $C_R(M)$ denotes the subchain of elements of $M \cap \F$ indexed by $R$. We say $M$ is \emph{$R$-bad}, if $C_R(M)$ is a bottleneck, and there exists a $(k,a)$-marked chain of $\LL$ whose $s$ largest markers are $C_R(M)$. Note that if $C_R(M) = \emptyset$, then $M$ is not $R$-bad, vacuously. Let $\st{B_R}\track{\text{Bad}_R}$ be the event that $M$ is $R$-bad. Our goal is to estimate the probability of $\st{B_R}\track{\text{Bad}_R}$ for each possible choice of $R$.

If $C_R(M)$ is not a bottleneck, then $\st{B_R}\track{\text{Bad}_R}$ does not happen. Thus suppose otherwise. Let $C_R(M)$ be a bottleneck with witness $\S := \S(C_R(M))$ and let $F \in \S$. As $\LL$ consists of $(k,a)$-marked chains, we can assume $|F'|\geq |F|+a$ for every $F\in \S$ and all $F' \in C_R(M)$. As  $|F| \ge \alpha n$,  it follows that, (recall that $M$ is randomly chosen, and $F$ and $\S$ are fixed),
\begin{align*}
    \P(F\in M \cap \S|\text{ }C_R(M) \text{ is a bottleneck with witness } \S)&\leq \binom{|F|+a}{|F|}^{-1} \leq \binom{\alpha n+a}{a}^{-1}\\
    &\leq \frac{a!}{(\alpha n)^{a}} = cn^{-a},
\end{align*}
as $c = a! \alpha^{-a}$.
As $|\S| \le |P|$, we have by a union bound,
\[\P(M \cap \S \neq \emptyset|\text{ }C_R(M) \text{ is a bottleneck with witness } \S) \le c|P|n^{-a}.\]

For a bottleneck $B$ of size $s$, let $E_B$ be the event that \st{$M$ is $R$-bad,} $C_R(M) = B$\st{,} and $B$ is a bottleneck with witness $\S(B)$. \track{ Recall, given $R\subseteq [K]$, $\text{Bad}_R$ denotes the event that $M$ is $R$-bad. Note that for $\text{Bad}_R$ to occur, $E_B$ has to occur for some bottleneck $B$.}

Let $\mathcal{B} \subseteq 2^{[n]}$ be the set of all bottlenecks of size $s$ such that $\P(E_B) > 0$. 
Note that if $B,B' \in \mathcal{B}$ such that $B \neq B'$, then $E_B$ and $E_{B'}$ are disjoint events. It follows, 
\begin{align*}
    \P(\st{B_R}\track{\text{Bad}_R}) &= \sum_{B \in \mathcal{B}}\P(E_B)\P(\st{B_R}\track{\text{Bad}_R}|E_B) \le \max_{B \in \mathcal{B}} \P(\st{B_R}\track{\text{Bad}_R} |E_B) \leq \max_{B \in \mathcal{B}}\P(M \cap \S(B) \neq \emptyset | E_B) \le  c |P|n^{-a}.
\end{align*}

Recall that as $|R| = s$, there are ${K \choose s}$ ways to select $R$ from $[K]$. Using $K/2 \ge k > s$\st{and $|P| < K^{-k}n^a/c$}, we have
 \[\P(M \text{ is bad})\leq \binom{K}{s}c|P|n^{-a} < \binom{K}{k}c|P|n^{-a} \le c\frac{|P|}{ k!  } K^k n^{-a} < 1\st{.}\track{,}\]
\track{where the last inequality follows from the hypothesis of Lemma \ref{lem:AuxSupersat1}, namely $|P| < K^{-k} n^{a}/c$.} Every bad chain $M$ corresponds to at most ${K \choose s} < {K \choose k}$ bad $(k,a)$-marked chains, hence we have the number of bad $(k,a)$-marked chains is at most
 \[ \frac{c |P|}{n^a}\binom{K}{k}^{2}n!.\]
Let $P' = P \setminus I$. By Lemma \ref{lem:decomposition}, $\{C_1,\ldots,C_{\ell - 1}\}$ is a graded chain cover of $P'$. Let $\LL'$ denote the set of good chains, then
\[ |\LL'| \ge |\LL| - \frac{c |P|}{n^a}{K \choose k}^2n! \ge \frac{c}{n^a}{K \choose k}^2\left( \binom{|P| + 1}{2} - |P| \right)n! = \frac{c}{n^a}{K\choose k}^2\binom{|P|}{2}n!.\]
In particular we have
\[ |\LL'| \ge \frac{c}{n^a}{K\choose k}^2\binom{|P'| + 1}{2}n!.\]

Thus by induction, there exists an embedding $\pi' : P' \to \F$ such that for every $i \in [\ell - 1]$, $\pi'(C_i)$ is the set of markers of some $(k,a)$-chain in $\LL'$. Let $F_1 \st{\subsetneq}\track{\supsetneq} \cdots \st{\subsetneq}\track{\supsetneq} F_s$ denote the markers corresponding to $C_\ell \setminus I$. As $\{F_1,\ldots,F_s\}$ is not a bottleneck and $C_\ell \setminus I \subseteq C_j \in \LL'$ for some $j \in [\ell - 1]$, there exists a $(k,a)$-marked chain $L \in \LL$ such that the first $s$ markers of $L$ correspond to $C_{\ell} \setminus I$ and its last $k - s$ markers do not intersect $\pi'(P \setminus I)$. We can thus extend $\pi'$ to an embedding $\pi : P \to \F$ such that $I$ is mapped to the last $k-s$ markers of $L$. Then $\pi(C_{\ell})$ is the set of markers of $L$, completing the proof.
\end{proof}

Let $\varepsilon > 0$, and as $\lim_{x \to \st{\infty}\track{0^{+}}}H(x) = 0$, there exists an $\st{\alpha}\track{\alpha:=\alpha(\varepsilon)} \in (0,1/2)$ such that $H(\alpha) < \epsilon$. Let $\F \subseteq 2^{[n]}$ and define
$$ \F_{\alpha} = \{F \in \F : |F| < \alpha n \text{ or } (1 - \alpha)n > |F|\}.$$
By Proposition \ref{prop:Entropy},
$$| \F_{\alpha} | \leq 2\binom{n}{\leq \alpha n}\leq 2^{H(\alpha)n+1}\leq 2^{\varepsilon n/2 + 1}\leq  \frac{\varepsilon}{2}\binom{n}{\lfloor n/2\rfloor},$$
for large enough $n$ with respect to $\varepsilon$.

\begin{cor}\label{cor:supersat1}
    Let $P$ be a graded tree poset of height $k$ and size $m$.  Suppose $x \in P$ and $\F\subseteq 2^{[n]}$ is a set system such that 
    \[
    |\mathcal{F}|>((k-1)(2k + 2m + 1)+\varepsilon)\binom{n}{\lfloor n/2\rfloor},
    \] 
    for some $\varepsilon > 0$. For $n$ sufficiently large (with respect to $m$, $k$, and $\varepsilon$), $\F$ contains a copy of $P(x,n)$.
\end{cor}

\begin{proof}
\st{Let} \track{Given $\varepsilon>0$, we fix $\alpha\in (0,1/2)$ such that $H(\alpha)<\varepsilon$. Let} 
   $$\F'=\{F\in \F : \alpha n\leq |F|\le (1-\alpha)n\}.$$
    For large enough $n$ with respect to $\varepsilon$, we have
  \[
    |\mathcal{F}'| = |\F \setminus \F_{\alpha}| >((k-1)(2k + 2m + 1)+\varepsilon/2)\binom{n}{\lfloor n/2\rfloor}.
    \]

   As $P$ is a graded poset of height $k$, by Lemma~\ref{lem:blowup_graded}, $P(x,n)$ is also a graded tree poset of height $k$.  As $|P| = m$, we have $|P(x,n)| \le mn^m$. Let $a = 2k + 2m + 1$, $K = n$, and $c = a! \alpha^{-a}$. Applying Lemma~\ref{lem:k_marked_chains_lowerbound} with $\F'$ and $a$,  there exists a family $\LL$ of $(k,a)$-marked chains such that $|\LL| \ge \frac{\varepsilon}{2k}n!$. We have then
    \begin{align*}
         \frac{c}{n^{2k+2m+1}}\binom{n}{k}^2\binom{mn^m+1}{2} = o(1) < \frac{\varepsilon}{2k},
    \end{align*}
     for sufficiently large $n$. Furthermore, for sufficiently large $n$, we have $K = n > 2k$, and 
     $$|P(x,n)| \le m n^m \le n^{a - k}/c = K^{-k}n^a/c.$$
     By Lemma~\ref{lem:decomposition}, $P(x,n)$ has a graded chain cover $\{C_1,\ldots,C_{\ell}\}$ for some $\ell$. By applying Lemma~\ref{lem:AuxSupersat1} to $\alpha$, $a$, $P(x,n)$, $\F'$, $\LL$, and $\{C_1,\ldots,C_{\ell}\}$ as above, for sufficiently large $n$, we may conclude $\F$ contains a copy of $P(x,n)$. 
\end{proof}

By an argument similar to the proof of Corollary~\ref{cor:supersat1}, we have the following.

\begin{cor}\label{cor:supersat2}
    Let $P$ be a graded tree poset of height $k$ of size $m$. Suppose $x \in P$ and $\F\subseteq 2^{[n]}$ such that
    \[|\mathcal{F}|>((k-1)+\varepsilon)\binom{n}{\lfloor n/2\rfloor}\] 
    for some $\varepsilon > 0$. For $n$ sufficiently large (with respect to  $m$, $k$, and $\varepsilon$), $\F$ contains a copy of $P(x,\log(n))$.
\end{cor}

\begin{proof}
\st{Let} \track{Given $\varepsilon>0$, we fix $\alpha\in (0,1/2)$ such that $H(\alpha)<\varepsilon$. Let} 
   $$\F'=\{F\in \F : \alpha n\leq |F|\le (1-\alpha)n\}.$$
    For large enough $n$ with respect to $\varepsilon$, we have 
  \[
    |\F'| = |\F \setminus \F_{\alpha}|>((k-1)+\varepsilon/2)\binom{n}{\lfloor n/2\rfloor}.
    \] 
    
        As $P$ is a graded poset of height $k$, by Lemma~\ref{lem:blowup_graded}, $P(x,\log n)$ is also a graded tree poset of height $k$.
     Furthermore, as $|P| = m$, we have $|P(x,\log n)| \le m(\log n)^m$. We may suppose no maximal chain of $2^{[n]}$ contains at least $m(\log n)^m$ sets from $\F'$, as otherwise, $\F'$ would contain a copy of $P(x,\log n)$. Let $a = 1$, $K =  m(\log n)^m$ and $c = a! \alpha^{-a}$. Applying Lemma~\ref{lem:k_marked_chains_lowerbound} with $\F\track{'}$ and $a$, we conclude that there exists a family $\LL$ of $(k,1)$-marked chains such that $|\LL| \ge \frac{\varepsilon}{2k} n!$. Then

    \begin{align*}
         \frac{c}{n}\binom{ m\log(n)^m}{k}^2\binom{m\log(n)^{m}+1}{2} = o(1) < \frac{\varepsilon}{2k},
    \end{align*}
    for sufficiently large $n$. By Lemma~\ref{lem:decomposition}, $P(x, \log n)$ has a graded chain cover, $\{C_1,\ldots,C_{\ell}\}$ for some $\ell$. \track{In order to apply Lemma \ref{lem:AuxSupersat1}, we need to show $K\geq 2k$ and $|P(x,\log n )| < K^{-k}n^{a}/c$.} For sufficiently large $n$, we have $K = m \log(n)^m > 2m \ge 2k$ and 
     $$|P(x,\track{\log n} \st{n})| \le m \log(n)^m \le (m \log(n)^m)^{-k} n/c = K^{-k}n^a/c.$$
    
     By applying Lemma \ref{lem:AuxSupersat1} to $\alpha$, $a$\track{, K}, $P(x,\track{\log n} \st{n})$, $\F'$, $\LL$, and $\{C_1,\ldots,C_{\ell}\}$ as above, for sufficiently large $n$, we conclude $\F$ contains a copy of $P(x,\track{\log n} \st{n})$. 
\end{proof}

\subsection{Container tools}

We will use a container algorithm based on the one used in \cite{PatkosTreglown}. We will run the algorithm in two phases, similarly to the application of the container algorithm used in  \cite{BaloghTwoPhase}. 

With $\F$ provided as an input, the goal of the algorithm is to build a unique pair $(\HH,\G)$ such that $\HH\subseteq \F\subseteq \HH\cup \G$, and $\G$ is determined by $\HH$. The key part of the proof of Theorem~\ref{thm:main} is that multiple $P$-free set systems are assigned to the same pair $(\HH,\G)$. To count all possible $P$-free set systems $\F$, first we count the number of possible pairs $(\HH,\G)$. With a pair $(\HH,\G)$ fixed, the number of subsets of $\G$ is an upper bound on the number of $P$-free set systems that are assigned to $(\HH,\G)$.

Historically, see \cite{BaloghMorrisSamotij,SaxtonThomas}, $\HH$ is called the \textit{certificate} (or \textit{fingerprint}) of $\F$ and $\HH\cup\G$ the \textit{container}. We are now ready to present our container algorithm.

\begin{lemma}\label{lem:algorithm}
    For every  tree poset $P$, an element $x\in P$, an integer $t$, and a set system $\S\subseteq 2^{[n]}$, there is a collection $\mathcal{C}$ of pairs $(\HH, \G)$ with $\HH,\G\subseteq \mathcal{S}$ such that:
    \begin{enumerate}[(i)]
        \item For each $P$-free set system $\F\subseteq \mathcal{S}$, there is a pair $(\HH, \mathcal{G})\in \mathcal{C}$ such that $\HH\subseteq \F\subseteq \HH\cup \G$.
        \item For each $(\HH,\G)\in \mathcal{C}$, $\G$ is $P(x,t)$-free and $|\HH|\leq |P||\S|/t$.
        \item For each $\HH\subseteq \mathcal{S}$ there is at most one $\G\subseteq \mathcal{S}$ such that $(\HH,\G)\in \mathcal{C}$. 
    \end{enumerate}
\end{lemma}

\begin{proof} Throughout this proof we will use three total orderings, a total ordering $\prec_1$ of the elements of $P$, i.e.,  $x=x_1\prec_1\ldots\prec_1 x_{|P|}$ given by Claim~\ref{claim:orderP}, the total ordering $\prec_2$ of the elements of $P(x,t)$ given by the lexicographic ordering as discussed in Definition~\ref{def:Blowup} and a total ordering $\prec_3$ of the copies of $P(x,t)$ in $2^{[n]}$.

We give a deterministic algorithm that takes as input a $P$-free set system $\F \subseteq \mathcal{S}$ and outputs the pair $(\HH,\G)$. 
 We initialize the variables $\HH_0=\emptyset$ and $\G_{0}=\S$. The following algorithm runs in possibly multiple iterations.
 
 We describe iteration $i \ge 0$ of the algorithm as follows. If $\G_i$ is $P(x,t)$-free, then the algorithm terminates and outputs $(\HH_i,\G_i)$. Otherwise $\G_i$ contains a copy of $P(x,t)$ and we let $\pi : P(x,t) \to \G_i$ denote a respective embedding. We further suppose $\pi(P(x,t))\subseteq \G_i$ is the first copy of $P(x,t)$ under the ordering $\prec_3$. With the given $(\HH_i,\G_i)$ and $\F$ we apply the following embedding procedure to obtain a new pair $(\HH_{i+1},\G_{i+1})$: 
\begin{itemize}
   
    \item[1.] If $\pi(x_{1,1})\notin \F$ then we set $\HH_{i+1}=\HH_i$ and $\G_{i+1}=\G_{i}\setminus \pi(x_{1,1})$ and we begin iteration $i+1$ with $(\HH_{i+1},\G_{i+1})$. 
    \item[2.] If $\pi(x_{1,1})\in \F$ then we let $Q_1 := \{\pi(x_{1,1})\}$.
    \item[3.] Let $Q_j := \{\pi (x_{1,r_1}),\ldots, \pi(x_{j,r_{j}})\}$ such that $Q_j$ is a subposet of $P$ induced by the first $j$ elements of the ordering $\prec_1$. Then $x_{j+1}$ is the first vertex of $\prec_1$ which is not embedded. By Claim~\ref{claim:orderP}, there is exactly one $x_k\in P$ such that $k \leq j$ and it is either the case that $x_{k}$ covers $x_{j+1}$ or $x_{\st{k}\track{j+1}}$ covers $x_{\st{j+1}\track{k}}$ in $P$. As $\pi(x_{k,r_k})\in Q_j$, we will try to embed one of its neighbors and thus grow $Q_j$. We break the process into two cases depending on $\pi(V_{j,r_k})\cap \F$. 
    
    \hspace{10mm}$3.1.$ If $\pi(V_{j,r_k})\cap \F\neq \emptyset$ then we denote by $x_{j+1,r_{j+1}}$ the first element under $\prec_2$ contained in $\pi(V_{j,r_k})\cap \F$. We let $Q_{j+1} := Q_j \cup \{\pi(x_{j+1,r_{j+1}})\}$ and we return to Step 3 with $Q_{j+1}$.
    
    \hspace{10mm}$3.2.$ If $\pi(V_{j,r_k})\cap \F =\emptyset$ then we let $\HH_{i + 1} := \HH_i \cup Q_j$ and $\G_{i + 1} := \G_i \setminus (\pi(V_{j,r_k})\cup Q_j)$. We begin iteration $i+1$ with $(\HH_{i+1},\G_{i+1})$.
   
\end{itemize}

Note, as $\F$ is $P$-free, Step 1 or Step 3.2 must occur once the above embedding procedure begins. Thus the embedding procedure, and consequently the algorithm, will eventually terminate for some pair $(\HH,\G)$ such that $\G$ is $P(x,t)$-free. As $\HH_0 = \emptyset$, and for all $i$, we only add elements to $\HH_i$ if they belong to $\F$, hence we have $\HH \subseteq \F$. Furthermore, as $\F \subseteq \G_0$, and for all $i$, the elements we remove from $\G_i$ are always disjoint from $\F$ or also added to $\HH_i$, we have $\HH \subseteq \F \subseteq \HH \cup \G$. Thus (i) holds.

For all $i$, if at least one element was added to $\HH_i$, i.e., if Step 3.2 occurs, then we also remove at least $t$ elements from $\G_i$. In particular, we have $|\HH| \le |P||\mathcal{S}|/t$ and (ii) holds as well.

 Let $\F$ and $\F'$ be two distinct $P$-free set systems such that the algorithm outputs $(\HH,\G)$ under input $\F$ and outputs $(\HH',\G')$ under input $\F'$. Suppose for the sake of a contradiction, that there exists $(\HH, \G),(\HH',\G')\in \mathcal{C}$ with $\HH=\HH'$ and $\G\neq \G'$. Then, there are two distinct $P$-free set systems $\F$ and $\F'$ such that each returns $(\HH,\G)$ and $(\HH',\G')$ when applying the algorithm. Since the algorithm is deterministic and $\G \neq \G'$, the algorithm at some iteration $i$, removes an element from $\G_i$ but not from $\G'_i$.
 Suppose this happens at Step 1 of the embedding procedure. Without loss of generality, suppose $\pi(x_{1,1}) \not \in \F$ and $\pi(x_{1,1}) \in \F'$. This would imply $\pi(x_{1,1}) \in \HH'$. As $\HH \subseteq \F$, we have $\HH \neq \HH'$, a contradiction. If this occurs with Step 3.2, by a similar argument, we obtain a contradiction. Thus we may conclude that (iii) holds, and consequently, the lemma holds as well.

\end{proof}

\begin{rem}\label{rem:sizeCont}
Since each choice of $\HH$ with $|\HH|\leq |P||\mathcal{S}|/t$ corresponds to at most one pair $(\HH,\G)\in\mathcal{C}$, we obtain
\[
|\mathcal{C}|\leq \binom{|\mathcal{S}|}{\leq |P||\mathcal{S}|/t}.
\]
\end{rem}

\noindent We will apply our container algorithm in two phases to obtain our main Container Lemma.
\begin{lemma}\label{lem:container} Let $P$ be a graded tree poset of height $k$. For each fixed $\varepsilon>0$, there is a collection $\mathcal{C}$ of pairs $(\HH, \G)$ with $\HH_,\G\subseteq 2^{[n]}$ and a function $\Psi$ that assigns each $P$-free set system $\F\subseteq 2^{[n]}$  a pair $\Psi(\F)=(\HH,\G)\in \mathcal{C}$ with the following properties. 
    \begin{enumerate}[(i)]
        \item $\HH\subseteq \F\subseteq \HH\cup \G$.
        \item $|\G|\leq (k-1+\varepsilon)\binom{n}{\lfloor n/2\rfloor}$.
        \item $|\mathcal{C}|=2^{O\left(\frac{\log\log n}{\log(n)}\binom{n}{\lfloor n/2 \rfloor}\right)}$.
    \end{enumerate}
\end{lemma}
\begin{proof}
We apply Lemma~\ref{lem:algorithm} with $P$, an arbitrary $x \in P$, $\mathcal{S}=2^{[n]}$ and $t=n$ to obtain a family of containers $\mathcal{C}_1$. From Remark~\ref{rem:sizeCont}, we get 
\begin{equation}\label{eq:FirstCont}
|\mathcal{C}_1|\leq \binom{2^{n}}{\leq|P|2^{n}/n}\leq 2^{H(|P|/n)2^{n}}= 2^{O\left(\log(n)2^{n}/n \right)},
\end{equation}
where the last inequality is obtained using Proposition~\ref{prop:Entropy}.

For each $(\HH_1,\G_1)\in \mathcal{C}_1$, we apply again Lemma~\ref{lem:algorithm}, this time with $P$, $x \in P$, $\mathcal{S}=\G_1$ and $t=\log(n)$, to obtain the family $\mathcal{C}_2(\G_1)$.  Finally we set 
$$\mathcal{C}:=\{(\HH,\G): \HH=\HH_1\cup\HH_2 \text{ and }\G=\G_2\text{ for some } (\HH_1,\G_1)\in\mathcal{C}_1, (\HH_2,\G_{2})\in\mathcal{C}_2(\G_1)\}.$$

Since for each $(\HH_1,\G_1)\in \mathcal{C}_1$ we have $\G_1$ is $P(x,n)$-free, using Corollary~\ref{cor:supersat1} we have that $|\G_1|\leq ((k-1)(2k + 2m + 1)+\varepsilon)\binom{n}{\lfloor n/2 \rfloor}$. Using Remark~\ref{rem:sizeCont}, we obtain 
\begin{equation}\label{eq:SecondCont}
|\mathcal{C}_2(\G_1)|\leq \binom{|\G_1|}{\leq |P||\G_1|/\log(n)}=2^{O\left(\frac{\log\log n}{\log(n)}\binom{n}{\lfloor n/2 \rfloor}\right)}.
\end{equation}
Thus, using a union bound for the size of $\mathcal{C}$ together with inequalities \eqref{eq:FirstCont} and \eqref{eq:SecondCont}, we have,
$$|\mathcal{C}|=|\cup_{(\HH_1,\G_1)\in \mathcal{C}_1}\mathcal{C}_2(\G_1)|\leq |\mathcal{C}_1|2^{O\left(\frac{\log\log n}{\log(n)}\binom{n}{\lfloor n/2 \rfloor}\right)}=2^{O\left(\frac{\log(n)2^{n}}{n} +\frac{\log\log n}{\log(n)}\binom{n}{\lfloor n/2 \rfloor}\right)}=2^{O\left(\frac{\log\log n}{\log(n)}\binom{n}{\lfloor n/2 \rfloor}\right)}.$$

For each $P$-free set system $\F\subseteq 2^{[n]}$ there is a pair $(\HH_1, \G_1)\in \mathcal{C}_1$ such that $\HH_{1}\subseteq \F\subseteq\HH_1\cup \G_1$. We have that $\F\cap \G_1$ is a $P$-free subfamily of $\G_1$ and thus, there is a pair $(\HH_2,\G_2)\in \mathcal{C}_{2}(\G_1)$ such that $\HH_2\subseteq \F\cap \G_1\subseteq \HH_2\cup\G_2$. Setting $\HH=\HH_1\cup\HH_2$ and $\G=\G_2$ we obtain the desired pair $(\HH,\G)$ with
\[
\HH=\HH_1\cup\HH_2\subseteq \F\subseteq \HH_1\cup\HH_2\cup\G_2=\HH\cup\G.
\]  
Since $\G_2$ is $P(x,\log(n))$-free, by Corollary~\ref{cor:supersat2} we have $|\G|=|\G_2|\leq (k-1+\varepsilon)\binom{n}{\lfloor n/2 \rfloor}$. 
\end{proof}

\section{Proof of Theorem~\ref{thm:main}}\label{sec:Proof}

Each subfamily of the $P$-free set system $\bigcup_{i = 0}^{k-2} \binom{[n]}{\floor{n/2} + i}$ is also $P$-free, this implies that the number of $P$-free set systems in $2^{[n]}$ is at least $2^{(1+o(1))(k-1)\binom{n}{\lfloor n/2\rfloor }}$. It only remains to show that the number of $P$-free set systems in $2^{[n]}$ is at most $ 2^{(1+o(1))(k-1)\binom{n}{\lfloor n/2\rfloor }}$. By Lemma \ref{lem:kRankedContention}, as $P$ is an induced subposet of a graded tree poset of height $k$, without loss of generality, we may assume $P$ is graded.

For fixed $\varepsilon>0$ we will show that, for large enough $n$, the number of $P$-free set systems is less than $2^{(k-1+\varepsilon)\binom{n}{\lfloor n/2 \rfloor}}$. We apply Lemma~\ref{lem:container}, with $\varepsilon/2$ as $\varepsilon$ and obtain our family of containers $\mathcal{C}$. For each $P$-free set system we obtain a pair $\Psi(\F)=(\HH,\G)\in \mathcal{C}$ such that $\HH\subseteq \F\subseteq \HH\cup\G$ and $|\G|\leq (k-1+\varepsilon/2)\binom{n}{\lfloor n/2\rfloor}$. Since $\HH\subseteq \F\subseteq \HH\cup \G$, the number of $P$-free set systems such that $\Psi(\F)=(\HH,\G)$ is at most the number of subsets of $\G$. Thus, the number of $P$-free set systems $\F\subseteq 2^{[n]}$ is at most 

\[
\sum_{(\HH,\G)\in \mathcal{C}}2^{|\G|}\leq |\mathcal{C}|2^{(k-1+\varepsilon/2)\binom{n}{\lfloor n/2 \rfloor}}\leq 2^{\left(k-1+\varepsilon/2+O\left(\frac{\log\log n}{\log n}\right)\right)\binom{n}{\lfloor n/2 \rfloor}}<2^{(k-1+\varepsilon)\binom{n}{\lfloor n/2\rfloor}},
\]
for large enough $n$, as desired.\\

\noindent \textbf{Acknowledgements} \track{The authors would like to  thank the anonymous referees for their careful reading and detailed suggestions.} \st{We}\track{The authors would also like to} thank Douglas West for the insightful conversation regarding terminology. \\

\noindent \textbf{Funding} J{\'o}zsef Balogh was supported in part by NSF grants DMS-1764123 and RTG DMS-1937241, FRG DMS-2152488, and the Arnold O. Beckman Research Award (UIUC Campus Research Board RB 24012). Michael Wigal was supported in part by NSF  RTG DMS-1937241.\\

\noindent \textbf{Code/Data Availability} Not applicable.

\section{Declarations}

\noindent \textbf{Conflict of Interest/Competing Interests} None.

\bibliographystyle{abbrv}
\bibliography{references} 

\end{document}